\documentclass{amsart}

\usepackage[utf8]{inputenc}
\usepackage{amsfonts}
\usepackage{amsmath}
\usepackage{amsthm}
\usepackage{amssymb}
\usepackage{latexsym}
\usepackage{enumerate}
\usepackage[pagewise]{lineno}
\usepackage[english]{babel}

\usepackage{amsmath}%
\usepackage{amsfonts}%
\usepackage{amssymb}%
\usepackage{amsthm}
\usepackage{graphicx}

\newtheorem{teorema}{Theorem}
\newtheorem*{teorema*}{Theorem}

\newtheorem{corolario}[teorema]{Corollary}
\newtheorem{definicao}[teorema]{Definition}

\newtheorem{lema}[teorema]{Lemma}
\newtheorem{notacao}[teorema]{Notation}

\numberwithin{equation}{section}

\begin{document}


\title{Two Cardinal inequalities about bidiscrete systems}

\author{Clayton Suguio Hida}
\thanks{The author was supported by FAPESP grants (2011/01285 - 3 and 2013/01609-9) and would like to thank Prof. Piotr Koszmider for his guidance and assistance during the preparation of this paper}

\address{Departamento de Matem\'atica, Instituto de Matem\'atica e Estat\'\i stica, Universidade de S\~ao Paulo,
Caixa Postal 66281, 05314-970, S\~ao Paulo, Brazil}
\email{suguio@ime.usp.br}

\begin{abstract}
We consider the cardinal invariant $bd$ defined by M. D{\v{z}}amonja and I. Juh{\'a}sz concerning bidiscrete systems. Using the relation between bidiscrete systems and irredundance for a compact Hausdorff space $K$, we prove that ${w(K)\leq bd(K)\cdot hL(K)^+}$, generalizing a result of S. Todorcevic concerning the irredundance in Boolean algebras and we prove that for every maximal irredundant family $\mathcal{F}\subset C(K)$, there is a $\pi$-base $\mathcal{B}$ for $K$ with $|\mathcal{F}|=|\mathcal{B}|$, a result analogous to the McKenzie Theorem for Boolean algebras in the context of compact spaces. In particular, it is a consequence of the latter result that $\pi(K)\leq bd(K)$ for every compact Hausdorff space $K$. From the relation between bidiscrete systems and biorthogonal systems, we obtain some results about biorthogonal systems in Banach spaces of the form $C(K)$.
\end{abstract}
\maketitle

\section*{Introduction}\label{funcoescardinais}
In this work, we consider the notion of bidiscrete systems defined by M. D{\v{z}}amonja and I. Juh{\'a}sz in \cite{dzmonja-juhasz-bidiscrete}:
\begin{definicao}[M. D{\v{z}}amonja, I. Juh{\'a}sz \cite{dzmonja-juhasz-bidiscrete}]Let $K$ be a compact Hausdorff space. A sequence $S=\{(x_\alpha^0,x_\alpha^1): \alpha<\kappa\}$ of pairs of points in $K$ (i.e. a subfamily of $K^2$) is called a bidiscrete system in $K$ if there exists a family $\{f_\alpha :\alpha<\kappa\}$ of real valued continuous functions on $K$ satisfying for every $\alpha,\beta<\kappa$:
\begin{itemize}
\item $f_\alpha(x_\alpha^l)=l$ for $l \in \{0,1\}$,
\item if $\alpha\neq \beta$ then $f_\alpha(x_\beta^0)=f_\alpha(x_\beta^1)$.
\end{itemize}

The cardinal invariant $bd(K)$ is defined to be 
$$bd(K):=\sup\{|S|: S \textrm{ is a bidiscrete system in } K\}.$$
\end{definicao}

The purpose of this work is to obtain some new cardinal inequalities for a compact Hausdorff space $K$, relating the cardinal invariant $bd(K)$ and some other topological cardinal invariants of $K$.
Firstly, we will translate the definition of $bd(K)$ in terms of Banach spaces and in terms of Banach algebras.
Given a compact Hausdorff space $K$, we consider the Banach space $C(K)$ of all real valued continuous functions on $K$ with the supremum norm. In \cite{dzmonja-juhasz-bidiscrete}, the authors considered the notion of a nice biorthogonal system for Banach spaces of the form $C(K)$. We say that a sequence $(f_\alpha, \mu_\alpha)_{\alpha<\kappa}$ in $C(K)\times M(K)$ (where $M(K)$ is the space of Radon measures on $K$) is a biorthogonal system if for every $\alpha, \beta <\kappa$
$$
\mu_\alpha(f_\beta) = \left\{
\begin{array}{rcl}
1,& \mbox{if} & \alpha=\beta\\
0, & \mbox{if} & \alpha\neq\beta.
\end{array}
\right.
$$
Moreover, if for each $\alpha<\kappa$, there are distinct points $x_\alpha, y_\alpha \in K$ such that ${\mu_\alpha=\delta_{x_\alpha}-\delta_{y_\alpha}}$, where $\delta_x$ denotes a Dirac measure centred in $x$, we say that the sequence $(f_\alpha, \mu_\alpha)_{\alpha<\kappa}$ is a nice biorthogonal system. 

We observe that, if  ${\{(x_\alpha,y_\alpha): \alpha<\kappa\}}$ is a bidiscrete system in $K$, then by definition, there exists a family of functions $\{f_\alpha :\alpha<\kappa\}$ such that ${(f_\alpha, \delta_{x_\alpha}-\delta_{y_\alpha})_{\alpha<\kappa}}$ is a nice biorthogonal system. In the same way, if ${(f_\alpha, \delta_{x_\alpha}-\delta_{y_\alpha})_{\alpha<\kappa}}$ is a nice biorthogonal system, then it is easy to see that ${\{(x_\alpha,y_\alpha): \alpha<\kappa\}}$ is a bidiscrete system. From this observation we conclude that 
$$bd(K)=\sup\{\kappa: \textrm{ there is a nice biorthogonal system of size } \kappa \textrm{ in } C(K)\}.$$

In other words, the cardinal $bd(K)$ is equal to the cardinal $nbiort_2(K)$ as defined in \cite{somepiotr11}.

For a compact Hausdorff space $K$, the corresponding Banach space $C(K)$ is a Banach algebra, where the multiplication is the pointwise multiplication of functions. A set $\mathcal{F}\subset C(K)$ is said to be irredundant if and only if for every $f \in \mathcal{F}$, there exists a Banach subalgebra $\mathcal{B}\subset C(K)$ such that $\mathcal{F}\setminus\{f\}\subset \mathcal{B}$ and $f\notin \mathcal{B}$. This is equivalent to say that, for every $f \in \mathcal{F}$, $f$ does not belong to the Banach subalgebra generated by $\mathcal{F}\setminus\{f\}$, where the Banach subalgebra generated by $\mathcal{F}\setminus\{f\}$ is the smallest Banach subalgebra of $C(K)$ containing $\mathcal{F}\setminus\{f\}$.
From Theorem 5.4 of \cite{somepiotr11}, for a compact Hausdorff space $K$, 
$C(K)$ contains an irredundant set of size $\kappa$ if and only if it contains a nice biorthogonal system of size $\kappa$. This implies that
$$bd(K)=\sup\{|\mathcal{F}|: \mathcal{F}\subset C(K) \textrm{ is an irredundant set of } C(K)\}.$$
We conclude that for every compact Hausdorff space $K$, there is a bidiscrete system in $K$ of size $\kappa$ if and only if there is an irredundant set of size $\kappa$ in $C(K)$. This allows us to work in the frame of irredundant sets instead of bidiscrete systems.
The definition of $bd(K)$ in terms of irredundant sets is similar to the definition of the well know cardinal invariant $irr(\mathcal{A})$ for a Boolean algebra $\mathcal{A}$. The cardinal invariant $irr(\mathcal{A})$ is defined in analogy to $bd$, where a subset $B\subset \mathcal{A}$ is irredundant if and only if for every $b \in B$, $b$ does not belong to the Boolean subalgebra generated by $B\setminus\{b\}$. See \cite{cardinalba} for definitions and some properties of $irr(\mathcal{A})$ for a Boolean algebra  $\mathcal{A}$. In particular, if $\mathcal{A}$ is a Boolean algebra and if $K_\mathcal{\mathcal{A}}$ denotes its Stone space, we have that $irr(\mathcal{A})\leq bd(K_\mathcal{A})$. In fact, suppose $B\subset \mathcal{A}$ is an irredundant set. Then $\mathcal{F}:=\{\chi_{[b]}:b \in B\}$ is an irredundant set in $C(K_\mathcal{A})$, where $[b]$ denotes the clopen set of $K_\mathcal{A}$ determined by $b \in B$. However, we do not know if we have the equality $irr(\mathcal{A})=bd(K_\mathcal{A})$. Theorem 5.4 of \cite{somepiotr11} tells us that for every Boolean algebra $\mathcal{A}$, $bd(K_\mathcal{A})\leq s(K_{\mathcal{A}}^2)$, where $s(K_\mathcal{A})$ is the spread of $K_\mathcal{A}$. This implies that if $\mathcal{A}$ is a counterexample for the equality $irr(\mathcal{A})=bd(K_\mathcal{A})$, then $irr(\mathcal{A})< s(K_{\mathcal{A}}^2)$. For example, the Boolean algebra $\mathcal{B}$ constructed in \cite{shelah-roslanowski} satisfies $irr(\mathcal{B})< s(K_{\mathcal{B}}^2)$ .

In \textbf{section 1}, we study the relation between the cardinal invariant $bd(K)$ and $w(K)$, the  topological weight of $K$. 
We prove the following result
\newtheorem*{teoremasbvar3}{Theorem \ref{teoremasbvar3}}
\begin{teoremasbvar3}
Let $K$ be a compact Hausdorff space. Then
$$w(K)\leq bd(K) \cdot hL(K)^+$$
where $w(K)$ is the topological weight of $K$ and $hL(K)$ is the hereditarily Lindelof degree of $K$.
\end{teoremasbvar3}

In \cite{irrstevo93}, S. Todorcevic proved that for every Boolean algebra $\mathcal{A}$,
$$|\mathcal{A}|\leq irr(\mathcal{A})\cdot ig(\mathcal{A})^+,$$
where $ig(\mathcal{A})$ is the minimal $\theta$ such that every ideal of $\mathcal{A}$ is generated by $\leq \theta$ elements. 
For a Boolean algebra $\mathcal{A}$, we have that $|\mathcal{A}|=w(K_\mathcal{A})$ and  ${irr(\mathcal{A})\leq bd(K_\mathcal{A})}$. Moreover, $ig(\mathcal{A})=hL(K_\mathcal{A})$. This follows from the fact that if $Y\subset K_\mathcal{A}$ is open, then $M:=\{a\in \mathcal{A}: [a]\subset Y\}$ is an ideal in $\mathcal{A}$. With these observations, we conclude that Theorem \ref{teoremasbvar3} is a generalization of the aforementioned Todocervic's result for Boolean algebras.

In \textbf{section 2}, we study the relation between the cardinal invariant $bd(K)$ and $\pi(K)$, the topological $\pi$-weight of $K$.
We prove the following result:

\newtheorem*{teoremaib3}{Theorem \ref{teoremaib3}}
\begin{teoremaib3}Let $K$ be a compact Hausdorff space. Then for every maximal irredundant set $\mathcal{F}$ in $C(K)$, there exists a $\pi$-base $\mathcal{B}$ of $K$ with $|\mathcal{B}|=|\mathcal{F}|$.
\end{teoremaib3}
The above result is a generalization of the McKenzie Theorem, which tells us that for every maximal irredundant set $B$ in a Boolean algebra $\mathcal{A}$, there is a dense subset $D$ of $\mathcal{A}$ with $|B|=|D|$. By a dense set in a Boolean algebra $\mathcal{A}$ we mean a subset $B$ with the property that for every element $a \in \mathcal{A}\setminus{0}$, there exists $b \in B$ such that $0<b\leq a$.  A reference for the McKenzie Theorem can be found in \cite{handbookboolean}, Proposition 4.23.
It is easy to see that, if $B$ is a dense set in a Boolean algebra $\mathcal{A}$, then $\mathcal{B}:=\{[b]: b \in B\}$ is a $\pi$-base for $K_\mathcal{A}$. It is important to mention that we cannot conclude Theorem \ref{teoremaib3} from the McKenzie Theorem in the special case of Boolean spaces because we do not know if a maximal irredundant set $\mathcal{F}$ in $C(K_\mathcal{A})$ gives us a maximal irredundant set in $\mathcal{A}$.

From Theorem \ref{teoremaib3}, we get the following result:
\begin{corolario}\label{corolarioteoremasb3}
For a compact Hausdorff space $K$, the following inequality holds:
$$\pi(K)\leq bd(K),$$
where $\pi(K)$ is the $\pi$-weight of $K$.
\end{corolario}

Using the fact that $bd(K)\leq s(K^2)$ (consequence of Theorem 5.4 and Theorem 5.5 in \cite{somepiotr11}) we get that 
\begin{corolario}\label{desigualdejuhasz}For a compact Hausdorff space $K$, the following inequality holds:
$$\pi(K)\leq s(K^2).$$
\end{corolario}
I. Juh{\'a}sz and Z. Szentmikl{\'o}ssy proved that for every compact Hausdorff space $K$, we have $d(K)\leq s(K^2)$ and $d(K)\leq bd(K)$ (Theorem 3 of \cite{juhaz} and Theorem 3.5 of \cite{dzmonja-juhasz-bidiscrete}). Using well knows results on cardinal invariants on a compact Hausdorff space $K$, the inequality in Corollary \ref{desigualdejuhasz} is a consequence of the inequality $d(K)\leq s(K^2)$, but the inequality in Corollary \ref{corolarioteoremasb3} seems to be stronger than the inequality $d(K)\leq bd(K)$.

For $X$ a regular space such that $w(X)$ is a strong cardinal, we have that $w(X)=\pi(X)$. From this and Corollary \ref{corolarioteoremasb3} we get the following result
\begin{corolario}\label{corolariownbiort}
If $K$ is a compact Hausdorff space such that $w(K)$ is a strong cardinal, we have that 
$$w(K)=bd(K).$$
\end{corolario}
In particular, in the context of biorthogonal systems, we have that $${w(K)=nbiort_2(K)}$$ for $K$ satisfying the condition in Corollary \ref{corolariownbiort}.


We finish this introduction with some remarks about the relation of the above results with semibiorthogonal systems in Banach spaces of the form $C(K)$.
\begin{definicao}Let $X$ be a Banach space and $X^*$ its dual space. A sequence $(x_\alpha, x^*_\alpha)_{\alpha <\kappa}$ in $X\times X^*$ is a semibiorthogonal system in $X$ if and only if for every $\alpha, \beta <\kappa$,we have:\begin{itemize}
\item $x^*_\alpha(x_\alpha)=1$,
\item $x^*_\alpha(x_\beta)=0$ if $\beta<\alpha$ and
\item $x^*_\alpha(x_\beta)\geq 0$ if $\beta>\alpha$.
\end{itemize}
We define $sbiort(X)$ as the supremum of cardinalities of semibiorthogonal systems on $X$.
\end{definicao}
In \cite{quocientstevo}, Theorem 9, S. Todorcevic proved that for any compact Hausdorff space $K$ of weight bigger than $\omega_1$, there is an uncountable semibiorthogonal system on $K$. With Theorem \ref{teoremasbvar3}, we can prove the following generalization of this result:
\begin{corolario}\label{corolarioquesta}Let $K$ be a compact Hausdorff space. Then
$$w(K)\leq sbiort(C(K))^+.$$
\end{corolario}
\begin{proof}
It is clear that a nice biorthogonal system is in particular a semibiorthogonal system. This implies that $$bd(K)\leq sbiort(C(K)).$$
For every cardinal $\kappa<hL(K)$, there is a semibiorthogonal sequence of size $\kappa$ (This is the content of Lazar's theorem. See \cite{lazar}). In particular, we have
$$hL(K)\leq sbiort(C(K)).$$
Then by Theorem \ref{teoremasbvar3} we conclude that
$$w(K)\leq bd(K) \cdot hL(K)^+\leq sbiort(C(K))^+.$$
\end{proof}
In \cite{somepiotr11}, Question 6.7, P. Koszmider asked if it is true that for every Banach space X, we have $d(X)\leq sbiort(X)^+$. Corollary \ref{corolarioquesta} gives a positive answer to this question in the case of Banach spaces of the form $C(K)$. 

The result of Corollary \ref{corolarioquesta} tells us that it is impossible to exist a Banach space of  the form $C(K)$ with big density and containing only small semibiorthogonal systems. In particular, there is no compact Hasdorff space of weight bigger than $\omega_1$ without uncountable semibiorthogonal systems (a result of S.Todorcevic). It is important to mention that this is not the case for biorthogonal systems, as shows the example of C. Brech and P. Koszmider in \cite{brechkoszmider} where they constructed a consistent example of a compact Hausdorff space $K$ of weight $\omega_2$ such that $C(K)$ does not have uncountable biorthogonal systems.

The notation and terminology of this paper are standard. In particular, a Boolean space is a totally disconnected compact Hausdorff space. For cardinal invariants, we use the results and notations of R. Hodel \cite{handbookconjuntos}. See the article \cite{somepiotr11} for the cardinal invariants related to biorthogonal systems.

\section{Bidiscrete systems and topological weight}\label{secao2fcardinais}
In this section, we prove a relation between 
the cardinal invariants $w(K)$, $hL(K)$ and $bd(K)$ for a compact Hausdorff space $K$. 

Following \cite{quocientstevo}, Theorem 9, we consider the following notation:
\begin{notacao}
Let $K$ be a compact Hausdorff space. For $f \in C(K)$ and $r \in \mathbb{R}$ we define
\begin{itemize}
\item $K(f\leq r):=f^{-1}((-\infty, r])$,
\item $K(f< r):=f^{-1}((-\infty, r))$,
\item $K(f\geq r) :=f^{-1}([r, +\infty))$ ,
\item $K(f> r) := f^{-1}((\infty, r))$.
\end{itemize}
\end{notacao}
We begin by proving an auxiliary lemma. 
\begin{lema}\label{lemasbvar4}
Let $K$ be a compact Hausdorff space of weight $\kappa$ such that $hL(K)^+<\kappa$. Then for every $\gamma<\kappa$, there exist
\begin{enumerate}[1)]
\item $\textrm{ }\mathcal{F}_\gamma \subset C(K),$
\item $\textrm{ }f_\gamma \in C(K),$
\item $\textrm{ } x_\gamma, y_\gamma \in K$
\end{enumerate}
such that
\begin{enumerate}[a)]

\item $\textrm{ }|\mathcal{F}_\gamma|<\kappa \textrm{ and } \mathcal{F}_\gamma \subset \mathcal{F}_{\gamma '}$  for every $\gamma\leq \gamma '$;
\item $\textrm{ } \mathcal{F}_\lambda = \cup_{\gamma<\lambda}\mathcal{F}_\gamma$ for $\lambda<\kappa$ a limit ordinal;
\item $\textrm{ } f_\gamma \in \mathcal{F}_{\gamma+1}$ and $f_\gamma(x_\gamma)=0, f_\gamma(y_\gamma)=1$ for every $\gamma<\kappa$;
\item For every $f \in \mathcal{F}_\gamma$,  $f(x_\gamma)=f(y_\gamma)$;
\item If $K(f\leq p)\cap K(g\geq q)=\emptyset$ for $f,g \in \mathcal{F}_\gamma$ and $p,q \in \mathbb{Q}$, then for every neighborhood $V,W$ of $x_\gamma$ and $y_\gamma$ respectively, the following hold: 
\begin{enumerate}[i)]
\item there is $h_1 \in \mathcal{F}_\gamma$ such that $h_1|_{K(f\leq p)}\equiv 0$ and $h_1|_{K(g\geq q)}\equiv 1$;
\item there is $h_2 \in \mathcal{F}_{\gamma+1}$ such that $x_\gamma \in K(h_2<1)\subset V$;
\item there is $h_3 \in \mathcal{F}_{\gamma+1}$ such that $y_\gamma \in K(h_3>0)\subset W$.
\end{enumerate}

\end{enumerate}
\end{lema}
\begin{proof}Let $K$ be a compact Hausdorff space of weight $\kappa$ such that $hL(K)^+<\kappa$.

Observe that, if $\mathcal{F}\subset C(K)$ and $|\mathcal{F}|< \kappa$, then $\mathcal{F}$ does not separate points of  $K$. Otherwise, we would have a continuous injection of $K$ into $[0,1]^{|\mathcal{F}|}$, a contradiction with the fact that $w(K)=\kappa>|\mathcal{F}|$.

Let us construct the sequences $(\mathcal{F}_\gamma)_{\gamma<\kappa}$, $(f_\gamma)_{\gamma<\kappa}$ and $(x_\gamma, y_\gamma)_{\gamma<\kappa}$ satisfying the lemma and such that
$$\forall \gamma<\kappa \ \ (|\mathcal{F}_\gamma|\leq \omega \cdot hL(K) \cdot|\gamma|)$$
The construction goes by induction on $\gamma<\kappa$. Define $\mathcal{F}_0:=\emptyset$.
Given $\gamma<\kappa$, suppose we have constructed $(\mathcal{F}_\alpha)_{\alpha \in \gamma}$, $(f_\alpha)_{\alpha<\gamma}$ and $(x_\alpha, y_\alpha)_{\alpha<\gamma}$.

If $\gamma$ is a limit ordinal, the construction is trivial.
Suppose now that $\beta = \gamma+1$ is a successor ordinal.  
By the above observation, as $|\mathcal{F}_\gamma|\leq \omega \cdot hL(K) \cdot|\gamma|< \kappa$, there exist $x_\gamma, y_\gamma \in K$ such that
$$\forall g \in \mathcal{F}_\gamma(g(x_\gamma)=g(y_\gamma)).$$
Consider $f_\gamma \in C(K)$ such that $f_\gamma(x_\gamma)=0$ and $f_\gamma(y_\gamma)=1$.
Now, the construction goes by a standard closure argument. Note that the character of a point $x \in K$ is not  bigger than $hL(K)$ because $L(K\setminus \{x\})\leq hL(K)$ and so $\{V\subset K\setminus\{x\}: V \textrm{ is open}\}$ has a subcover $\mathcal{U}$ with $|\mathcal{U}|\leq hL(K)$. Now $\bigcap \{K\setminus \overline{V}: V \in \mathcal{U}\}=\{x\}$ and so the pseudocharacter (which is equal to the character for compact Hausdorff spaces) of $x$ is not bigger than $hL(K)$. This implies that, for every $x \in K$, we can consider a family $\mathcal{F}:=\{g_\alpha, \alpha<hL(K)\}\subset C(K)$ such that $\{K(g_\alpha<1): \alpha<hL(K)\}$ is a base for the point $x$.
\end{proof}

Observe that, in the previous lemma, the family $S:=\{(x_\alpha, y_\alpha): \alpha<\kappa\}$ satisfies the condition to be a bidiscrete system except from the fact that we do not know if $f_\alpha(x_\beta)=f_\alpha(y_\beta)$ for $\beta<\alpha$. To get this property, we will refine the set $S$. This will be the idea of the following theorem:
\index{$w(K)$}\index{$nbiort_2(K)$}\index{$hL(K)$}

\begin{teorema}\label{teoremasbvar3}
Let $K$ be a compact Hausdorff space. Then
$$w(K)\leq bd(K) \cdot hL(K)^+.$$
\end{teorema}
\begin{proof}
We may assume that $hL(K)^+<w(K)$. Let $\lambda$ be a regular cardinal such that $hL(K)^+ <\lambda \leq w(K)$. We will construct a bidiscrete system of cardinality $\lambda$ and, as $\lambda$ is arbitrary, this completes the proof.
Consider sequences $(\mathcal{F}_\gamma)_{\gamma<\kappa}$, $(f_\gamma)_{\gamma<\kappa}$ and $(x_\gamma, y_\gamma)_{\gamma<\kappa}$ as in Lemma \ref{lemasbvar4}.
For each $\gamma<\lambda$, define
$$\mathcal{U}_\gamma := \{K(f<p), K(g>q): f,g \in \mathcal{F}_\gamma, p,q \in \mathbb{Q}\}.$$

Observe that, since $(\mathcal{F}_\gamma)_{\gamma<\lambda}$ is a continuous and increasing family, $(\mathcal{U}_\gamma)_{\gamma<\lambda}$ is continuous and increasing as well.

Fix $\gamma<\lambda$ of cofinality $hL(K)^+$. For every $\gamma_1, \gamma_2 <\lambda$ and $p \in \mathbb{Q}$ we define 
$$\mathcal{B}_{\gamma_1,\gamma_2}(<p) := \bigcup\{V \in \mathcal{U}_{\gamma_1}: V \subset K(f_{\gamma_2}<p)\}\textrm{ and }$$
$$\mathcal{B}_{\gamma_1,\gamma_2}(>p) := \bigcup\{V \in \mathcal{U}_{\gamma_1}: V \subset K(f_{\gamma_2}>p)\}.$$

\textbf{Claim 1.} \textit{There exists $\xi_\gamma<\gamma$ such that for each $p,q \in \mathbb{Q}$}
$$\mathcal{B}_{\gamma,\gamma}(<p) = \mathcal{B}_{\gamma, \xi_\gamma}(<p) \textrm{ and }\mathcal{B}_{\gamma,\gamma}(>q) = \mathcal{B}_{\gamma, \xi_\gamma}(>q).$$
\begin{proof}[Proof of Claim 1.]
Since $\gamma$ is a limit ordinal, we can write for every $p\in \mathbb{Q}$:
$$\mathcal{B}_{\gamma,\gamma}(<p) = \bigcup_{\alpha<\gamma}\mathcal{B}_{\alpha,\gamma}(<p).$$
By definition of $hL(K)$, there exists $\Gamma_p \subset \gamma$ of cardinality at most $hL(K)$ such that
$$\mathcal{B}_{\gamma,\gamma}(<p) = \bigcup_{\alpha\in \Gamma_p}\mathcal{B}_{\alpha,\gamma}(<p).$$
As $|\Gamma_p|\leq hL(K)$ and $\gamma$ has cofinality $hL(K)^+$, there exists $\beta_p \in \gamma$ such that $\sup \Gamma_p \leq \beta_p$. Then
$$\mathcal{B}_{\gamma,\gamma}(<p) = \bigcup_{\alpha \in \Gamma_p}\mathcal{B}_{\alpha,\gamma}(<p)=\mathcal{B}_{\beta_p,\gamma}(<p).$$
Consider the mapping $\beta:\mathbb{Q}\to \gamma$ given by $\beta(p)=\beta_p$.
As $cf(\gamma)=hL(K)^+$, there exists $\phi<\gamma$ such that for every $p \in \mathbb{Q}$ we have that $\beta_p<\phi$.
Then for each $p \in \mathbb{Q}$,
$$\mathcal{B}_{\gamma,\gamma}(<p)=\mathcal{B}_{\phi,\gamma}(<p).$$
In a similar way, we can find $\phi'<\gamma$ such that for every $p \in \mathbb{Q}$,
$$\mathcal{B}_{\gamma,\gamma}(>p)=\mathcal{B}_{\phi,\gamma}(>p).$$
To conclude de proof of Claim 1, define $\xi_\gamma=\max\{\phi, \phi'\}$.
\end{proof}

Consider $S:=\{\gamma<\lambda: cf(\gamma)=hL(K)^+\}$. By Lemma 6.10 in \cite{kunen}, the set $S$ is stationary in $\lambda$ and by Claim 1, we have a regressive function $f:S\to \lambda$, given by $f(\gamma)=\xi_\gamma$.
By Fodor's theorem (Lemma 6.15 in \cite{kunen}), there exists  a stationary set $\Gamma \subset S$ and $\beta<\min \Gamma$ such that $f_{|\Gamma}\equiv \beta$. Then for each $\gamma \in \Gamma$ and for every $p,q \in \mathbb{Q}$,
$$\mathcal{B}_{\gamma,\gamma}(<p)=\mathcal{B}_{\beta,\gamma}(<p) \textrm{ and } \mathcal{B}_{\gamma,\gamma}(>q)=\mathcal{B}_{\beta,\gamma}(>q).$$

\textbf{Claim 2.} \textit{$f_\delta(x_\gamma)=f_\delta(y_\gamma)$ for every $\delta, \gamma \in \Gamma$ with $\gamma<\delta$.}
\begin{proof}[Proof of Claim 2.]
Suppose the opposite, i.e, there exist $\delta, \gamma \in \Gamma$ with $\gamma<\delta$ and such that $f_\delta(x_\gamma)\neq f_\delta(y_\gamma)$. Then there exist $p<q$ such that\\ 
I) $f_\delta(x_\gamma)<p<q<f_\delta(y_\gamma)$ or\\
II) $f_\delta(y_\gamma)<p<q<f_\delta(x_\gamma)$.\\
Suppose that I) holds. We have that $x_\gamma \in K(f_\delta < p)$ and $y_\gamma \in K(f_\delta>q)$. By Lemma \ref{lemasbvar4}, there exist $f,g \in \mathcal{F}_{\gamma+1}\subset \mathcal{F}_\delta$ such that
$$x_\gamma \in K(f<1)\subset  K(f_\delta < p) \textrm{ and } y_\gamma \in K(g<1)\subset  K(f_\delta >q).$$
By the definition of $\mathcal{U}_\delta$, we have that $K(f<1), K(g<1) \in \mathcal{U}_\delta$.
Then $x_\gamma \in K(f<1) \subset \bigcup\{V \in \mathcal{U}_\delta: V\subset K(f_\delta<p)\}$.
By the choice of $\beta$,
$$\bigcup\{V \in \mathcal{U}_\delta: V\subset K(f_\delta<p)\} = \bigcup \{V \in \mathcal{U}_\delta: V\subset K(f_\delta<p)\}.$$

Therefore, there exists $V \in \mathcal{U}_\beta$ such that $x_\gamma \in V \subset K(f_\delta<p).$
In a similar way, there exists $W \in \mathcal{U}_\beta$ such that $y_\gamma \in W \subset K(f_\delta>q).$

As $V,W \in \mathcal{U}_\beta$, there exists $g,h \in \mathcal{F}_\beta$ and $t,v \in \mathbb{Q}$ such that $V=K(g<t)$ and $W=K(h>v)$. Then
$$x_\gamma \in K(g<t)\subset K(f_\delta<p) \textrm{ and } y_\gamma \in K(h>v)\subset K(f_\delta>q).$$
We can assume that $K(g\leq t)\subset K(f_\delta<p)$ and $K(h\geq v)\subset K(f_\delta>q)$.
Then $$K(g\leq t)\cap K(h\geq v)\subset K(f_\delta<p)\cap K(f_\delta>q)=\emptyset.$$
By Lemma \ref{lemasbvar4}, there exists $h \in \mathcal{F}_\beta$ such that
$$f_{|K(g\leq t)}\equiv 0 \textrm{ and } f_{|K(h\geq v)}\equiv 1.$$
In particular, $f(x_\gamma)=0\neq 1 = f(y_\gamma)$. However, $\beta<\min \Gamma\leq \gamma$ which contradicts the fact that $g(x_\gamma)=g(y_\gamma)$ for every $g \in \mathcal{F}_\sigma$ with $\sigma<\gamma$.

This finishes the proof of Claim 2.
\end{proof}
Consider now the sequence $S:=(x_\gamma, y_\gamma)_{\gamma \in \Gamma}$. Given $\delta, \gamma$ in $\Gamma$, if $\delta<\gamma$ then $f_\gamma(x_\delta)=f_\gamma(x_\delta)$ (By Lemma\ref{lemasbvar4}) and if $\gamma<\delta$ then $f_\delta(x_\gamma)=f_\delta(x_\gamma)$ (By Claim 2). This proves that $S$ is a bidiscrete system. Therefore, $\lambda=|\Gamma|\leq bd(K)$. As $\lambda$ was an arbitrary regular cardinal smaller than $w(K)$, we conclude that $w(K)\leq bd(K)$.
\end{proof}

\section{Bidiscrete systems and $\pi$-bases}\label{secao3fcardinais}
In this section, we consider the relation between irredundance and $\pi$-bases.
The McKenzie Theorem says that for every Boolean algebra $\mathcal{A}$ and for every maximal irredundant set $\mathcal{F}$ in $\mathcal{A}$, there exists a dense subset $B \subset \mathcal{A}$ with $|B|=|\mathcal{F}|$ (see \cite{handbookboolean}, Proposition 4.23).
We will prove a result analogous to the McKenzie Theorem in the context of compact Hausdorff spaces, not necessarily a Boolean space. This is the content of the following theorem:
\begin{teorema}\label{teoremaib3}
For every maximal irredundant set $\mathcal{F}$ in a Banach space $C(K)$, there exists a $\pi$-base $\mathcal{B}$ of $K$ with $|\mathcal{B}|=|\mathcal{F}|$.
\end{teorema}
\begin{proof}
Let $\mathcal{F}\subset C(K)$ be a maximal irredundant set.
We will denote by $\langle \mathcal{F} \rangle_N$ the set of all finite linear combinations of finite products of elements in $\mathcal{F}$ with coefficients in $N$. In the proof, we consider $N=\mathbb{Q}$ or $N=\mathbb{R}$.

For each $g \in C(K)$ and $p \in \mathbb{Q}^+$ define the open set
$$B(g,p):=g^{-1}((-\infty, -p)\cup(p, \infty)).$$
Consider $\mathcal{B}:=\{B(g,p): p \in \mathbb{Q}^+, g \in \langle \mathcal{F} \rangle_{\mathbb{Q}}\}.$
We have that $|\mathcal{B}|\leq |\mathcal{F}|$. Let us see that $\mathcal{B}$ is a $\pi$-base for $K$.
Suppose that $U$ is a nonempty open set of $K$. Consider $f_U \in C(K)$ such that $f_U$ vanishes outside of $U$ and takes the value $1$ in some point of $U$.
We have two cases to consider: if $f_U \in \mathcal{F}$ or $f_U \notin \mathcal{F}$.

\textbf{Case 1.} If $f_U \in \mathcal{F}$, then $f_U \in \langle \mathcal{F} \rangle_{\mathbb{Q}}$ and ${B(f_U,\frac{1}{2})=f_U^{-1}((-\infty, -\frac{1}{2})\bigcup(\frac{1}{2}, \infty))\in \mathcal{B}}$ is a non empty open set contained in $U$.

\textbf{Case 2.} Suppose now that $f_U \notin \mathcal{F}$.
By the maximality of $\mathcal{F}$, the set $\mathcal{F}\cup\{f_U\}$ is not an irredundant set. Then $f_U \in \overline{\langle \mathcal{F} \rangle}_\mathbb{R}$ or there exists $g \in \mathcal{F}$ such that $g \in  \overline{\langle \mathcal{F}\setminus\{g\}\cup \{f_U\} \rangle_\mathbb{R}}$.

\textbf{Subcase 1.} If $f_U \in \overline{\langle \mathcal{F} \rangle}_\mathbb{R}$, as $\langle \mathcal{F} \rangle_{\mathbb{Q}}$ is dense in $\overline{\langle \mathcal{F} \rangle}_\mathbb{R}$, there exists $g \in \langle \mathcal{F} \rangle_{\mathbb{Q}}$ such that $||f_U-g|| < \frac{1}{2}$.
Then for each $x \in K\setminus U$, we have that $$|g(x)|=|f_U(x)-g(x)|\leq ||f_U-g|| < \frac{1}{2}$$ and therefore,
$$B(g,\frac{1}{2})=g^{-1}((-\infty, -\frac{1}{2})\cup(\frac{1}{2}, \infty))\subset U.$$
Observe that, by hypothesis, there exists $x \in U$ such that $f_U(x)=1$. Then $|f_U(x)-g(x)|=|1-g(x)|<\frac{1}{2}$ and therefore $|g(x)|> \frac{1}{2}$. This shows us that the open set $B(g,\frac{1}{2})$ is not empty.

\textbf{Subcase 2.} Suppose now that there exists $g \in \mathcal{F}$ such that $g \in  \overline{\langle \mathcal{F}\cup \{f_U\}\setminus\{g\}\rangle}_\mathbb{R}$.
In this case, there are sequences $(b_n)_{n\in \mathbb{N}}$ in $\langle \mathcal{F}\setminus\{g\} \rangle_{\mathbb{Q}}$ and $(c_n)_{n\in \mathbb{N}}$ in ${\langle \mathcal{F}\cup \{f_U\}\setminus\{g\}\rangle_{\mathbb{Q}}}$ such that the sequence $(b_n+c_n f_U)_{n\in \mathbb{N}}$ converges to $g$ in $C(K)$.

\textbf{Claim 1: } \textit{Given $p \in \mathbb{Q}^+$, there exists $n_0 \in \omega$ such that for all $n>n_0$
$$B(g-b_n,p)\in \mathcal{B} \textrm{ and } B(g-b_n,p)\subset U.$$}
\begin{proof}[Proof of Claim 1]
It follows from the definition of convergence that there exists $n_0 \in \omega$ such that for each $n>n_0$ we have that
$$||g-(b_n+c_n f_U)||=\sup_{x \in K}|g(x)-(b_n(x)+(c_n f_U)(x))|<p.$$
In particular, if $x \in K\setminus U$ and $n>n_0$  (remembering that $f_U$ vanishes outside of $U$), then $$|g(x)-b_n(x)|=|g(x)-b_n(x)-(c_n f_U)(x)|\leq ||g-(b_n+c_n f_U)||<p.$$
Therefore, if $n>n_0$ we have that $B(g-b_n,p)\subset U$ and because $g-b_n \in  \langle \mathcal{F} \rangle_{\mathbb{Q}}$, it follows that $B(g-b_n,p) \in \mathcal{B}$. 
\end{proof}
From Claim 1, it is enough to find an element $q \in\mathbb{Q}^+$ such that $B(g-b_n,q)$ is nonempty for arbitrarily large $n$.

\textbf{Claim 2}: \textit{There exists $q \in \mathbb{Q}^+$ such that for each $n_1 \in \omega$ there exists $n>n_1$ such that $B(g-b_n,q)$ is nonempty.}
\begin{proof}[Proof of Claim 2]
Since $\mathcal{F}$ is irredundant, the sequence of functions $(b_n)_{n\in \mathbb{N}}$ does not converge to $g$ in $C(K)$. Then there exists $\varepsilon>0$ such that for each $n_1 \in \omega$, there exists $n>n_1$ such that
$$\sup_{x \in K}|g(x)-b_n(x)|>\epsilon.$$
Then there exists $y \in K$ such that $|g(y)-b_n(y)|=\sup_{x \in K}|g(x)-b_n(x)|>\epsilon$.
Let $q \in \mathbb{Q}^+$ be such that $q<\epsilon$. Then $y \in B(g-b_n,q)$ and therefore, $B(g-b_n,q)$ is nonempty.
\end{proof}
To conclude the proof, fix $q \in \mathbb{Q}^+$ as in Claim 2. By Claim 1, there exists $n_0 \in \omega$ such that for all $n>n_0$
$$B(g-b_n,q)\in \mathcal{B} \textrm{ and } B(g-b_n,q)\subset U.$$
By the choice of $q$, there exists $n_1>n_0$ such that $B(g-b_{n_1},q)\neq \emptyset$. Then 
$$\emptyset \neq B(g-b_{n_1},q)\subset U.$$
\end{proof}

\bibliographystyle{amsplain}

\end{document}